\documentclass[12pt,a4paper]{amsart}

\usepackage{amsmath,amsfonts,amssymb,mathtools,extarrows}

\usepackage[hmargin=2.5cm,vmargin=2.5cm]{geometry}
\usepackage[dvipsnames]{xcolor}	
\usepackage{hyperref}
\usepackage{nameref,zref-xr}     

\hypersetup{
colorlinks=true, linktocpage=true, pdfstartpage=1, pdfstartview=FitV,breaklinks=true, pdfpagemode=UseNone, pageanchor=true, pdfpagemode=UseOutlines,plainpages=false, bookmarksnumbered, bookmarksopen=true, bookmarksopenlevel=1,hypertexnames=true, pdfhighlight=/O,urlcolor=webbrown, linkcolor=RoyalBlue, citecolor=ForestGreen}

\usepackage{enumerate}      

\usepackage{scalerel}         
\usepackage{stmaryrd}

\usepackage[all]{xy}

\usepackage[textsize=footnotesize,textwidth=0.85in]{todonotes}
\presetkeys%
    {todonotes}%
    {color=Apricot}{}%
\newcommand{\oM}{\overline{\mathcal M}}

\def\oM{{\overline{\mathcal{M}}}}

\newcommand{\ch}{\mathrm{ch}}

\newcommand{\DR}{\mathrm{DR}}

\newcommand{\SRT}{\mathrm{SRT}}
\newcommand{\DLSRT}{\mathrm{DLSRT}}

\newcommand{\ZZ}{\mathbb{Z}}
\newcommand{\QQ}{\mathbb{Q}}

\DeclareMathOperator{\Aut}{Aut}
\newcommand{\Ch}{\Omega}

\newtheorem{theorem}{Theorem}[section]
\newtheorem{proposition}[theorem]{Proposition}
\newtheorem{lemma}[theorem]{Lemma}
\newtheorem{corollary}[theorem]{Corollary}
\newtheorem{conjecture}[theorem]{Conjecture}

\theoremstyle{remark}

\newtheorem{remark}[theorem]{Remark}

\theoremstyle{definition}

\newtheorem{definition}[theorem]{Definition}

\usepackage{color}

\numberwithin{equation}{section}

\begin{document}

\title[Stable tree expressions with Omega-classes and DR cycles]{Stable tree expressions with Omega-classes and double ramification cycles}

\author[X.~Blot]{Xavier Blot}
\address{X.~B.: Korteweg-de Vriesinstituut voor Wiskunde, Universiteit van Amsterdam, Postbus 94248, 1090GE Amsterdam, Nederland}
\email{x.j.c.v.blot@uva.nl}	

\author[D.~Lewa\'nski]{Danilo Lewa\'nski}
\address{D.~L.: Dipartimento di Matematica, Informatica e Geoscienze, Universit\`a degli studi di Trieste,
	Via Weiss 2, 34128	Trieste, Italia}
\email{danilo.lewanski@units.it} 

\author[P.~Rossi]{Paolo Rossi}
\address{P.~R.: Dipartimento di Matematica``Tullio Levi-Civita'', Universit\`a degli studi di Padova,
	Via Trieste 63, 35121 Padova, Italia}
\email{paolo.rossi@math.unipd.it}

\author[S.~Shadrin]{Sergei Shadrin}
\address{S.~S.: Korteweg-de Vriesinstituut voor Wiskunde, Universiteit van Amsterdam, Postbus 94248, 1090GE Amsterdam, Nederland}
\email{s.shadrin@uva.nl}	

\begin{abstract} We propose a new system of conjectural relations in the tautological ring of the moduli space of curves involving stable rooted trees with level structure decorated by Hodge and $\Omega$-classes and prove these conjectures in different cases.
\end{abstract}

\maketitle
\tableofcontents

\section{Introduction} The goal of this paper is to present a general conjecture and a number of proven special cases that express the so-called $A$-classes emerging in the theory of the DR/DZ correspondence for integrable hierarchies \cite{BDGR1,BGR19} in terms of stable trees decorated by products of the $\lambda$-classes and the so-called $\Omega$-classes, as well as the vanishing identities of the similar flavor. The results and conjectures that we present here are to a large extent parallel to the ones appearing in~\cite{BS22}. The main differences lie in the following two aspects:
\begin{description}
	\item[interpretation] The results and conjectures of~\cite{BS22} aim to proving the existence and equivalence statements in the theory of DR/DZ correspondence, while in our approach the goal is to explain the surprising numerical observations made by the first and the second named authors for the tau functions of the DR hierarchy.
	\item[decorations] The decoration of the stable trees involved in respective the tautological relations is also quite different: in~\cite{BS22} stable trees are always decorated by $\psi$-classes, while here we use products of $\lambda$-classes and $\Omega$-classes, as mentioned above.\\
\end{description}

The motivation for the present work comes from two different aspects in the study of the so-called double ramification (DR) hierarchies of PDEs associated to cohomological field theories, intorduced by Buryak in~\cite{Bur} and extended to the case of F-cohomological field theories in \cite{BR21,ABLR21}. On the one hand, there is a program of establishing a Miura equivalence between the DR hierarchies and Dubrovin-Zhang hierarchies for a large class of F-cohomological field theories (as well as polynomiality of conservation laws of the latter): these facts were reduced in~\cite{BGR19,BS22} to a system of tautological relations of a type that had never emerged in the literature before. In this direction, our paper can be considered as an enhancement of the system of tautological relations of this type, with new examples. We hope that, once a critical mass of conjectural and proven relations of this kind is known, this can lead to their full understanding and general proof.\\

On the other hand, the study of quantum tau functions, introduced in~\cite{BDGR1} in the context of the quantization of DR hierarchies, and performed in~\cite{Blot} and in a forthcoming paper~\cite{BlotLew}, strongly suggests a system of relations between double ramification cycles and $\Omega$-classes, with no immediate resemblance with the known expression for the double ramification cycles proved in~\cite{JPPZ}. Furthermore, restricting the conjecture of~\cite{BlotLew} to the classical setting suggests a systems of relations between the A-classes of \cite{BGR19,BS22} and $\Omega$-classes. In this sense, the paper presents a system of relations revealing what is behind the numerical observations of~\cite{BlotLew} in the classical setting.\\

As a final remark let us mention that the $\Omega$-classes that we use here recently became ubiquitous in the interaction of enumerative geometry and integrability. They are used in a variety of works on Hurwitz numbers, Mirror Symmetry of toric Calaby-Yau threefolds, Masur-Veech volumes, $\Theta$-classes and their relation to integrable hierarchies, Euler characteristic of the open moduli spaces of curves, spin Gromov-Witten theory of curves and the study of spaces of meromorphic differentials on algebraic curves, in the expressions of the double ramification cycles and their different variations, among other topics. See for instance \cite{rELSV, JPPZ, BKL, spin, Euler, Double} and references therein.\\

\subsection{Organization of the paper} The paper is not entirely self-contained, as we borrow several constructions and ideas from a number of previous works on similar structures, see~\cite{BGR19,BHS22} and especially~\cite{AS09} and~\cite{BS22}; the material presented in these last two papers is necessary to follow our arguments. We assume the reader to be familiar with the standard notation for various classes in the tautological ring of the moduli spaces of curves and the usage of stable graphs for their expressions.\\

We present our results and conjectures in Section~\ref{sec:results} assuming the reader is familiar not only with the standard stratification of the moduli spaces of curves and $\psi$-, $\kappa$-, and $\lambda$-classes in the tautological ring, but also with less known $\Omega$-classes. The introduction to the $\Omega$-classes and their properties is postponed to Section~\ref{sec:Omega:classes}, where we recall their basic properties and use them to analyze the formulas proposed in Section~\ref{sec:results}. Finally, in Section~\ref{sec:Arcara:Sato} we employ localization techniques along the lines of the work of Arcara-Sato~\cite{AS09} in order to prove the main results of this paper. \\

\subsection{Conventions and notation} We work with classes in the tautological ring of the moduli space of stable curves, $R^*(\oM_{g,n})$; when we say ``cohomological degree'' we abuse the terminology and actually refer to the degree in the Chow ring, so half the actual degree in cohomology.  For a tautological class $c$ of mixed degree we denote by $(c)_i$ its homogenenous component of degree $i$ in Chow and $2i$ in cohomology. 

\subsection{Acknowledgments} This project started at a conference on moduli spaces in Leysin in March 2022 supported by the SNSF Ambizione grant $PZ00P2/202123$. We express our gratitude to the developers of the \texttt{admcycles} package~\cite{admcycles}, as the experiments leading to the conjectures presented in this paper were conducted with this package.

X.~B. was partially supported by the Netherlands Organization for Scientific Research, by the FCT project UIDB/00208/2020 and by the ISF grant 335/19 in the group of Ran Tessler.
D.~L. has been supported by the SNSF Ambizione grant $PZ00P2/202123$ hosted at Section de Math\'{e}matiques de l'Universit\'{e} de Gen\`{e}ve, by the University of Trieste, by the INFN under the national project MMNLP and by the INdAM group GNSAGA.
P.~R. is supported by the University of Padova, by the INFN under the national project MMNLP and by the INdAM group GNSAGA.
S.~S. is supported by the Netherlands Organization for Scientific Research.

\section{Conjectures and theorems} \label{sec:results}

\subsection{Basic notation for trees}
Let $\SRT_{g,n,m}$ be the set of stable rooted trees of total genus $g$, with $n$ regular legs $\sigma_1,\dots,\sigma_n$ and $m$ extra legs $\sigma_{n+1},\dots,\sigma_{n+m}$, which we refer to as ``frozen'' legs and must always be attached to the root vertex. For a $T\in \SRT_{g,n,m}$ we use the following notation:
\begin{itemize}
	\item $H(T)$ is the set of half-edges of $T$.
	\item $L(T),L_r(T),L_f(T)\subset H(T)$ are the sets of all, regular, and frozen legs of $T$, respectively. $L(T) = L_r(T)\sqcup L_f(T)$.
	\item $H_e(T)\coloneqq H(T)\setminus L(T)$.
	\item $\iota\colon H_e(T)\to H_e(T)$ is the involution that interchanges the half-edges that form an edge.
	\item $E(T)$ is the set of edges of $T$, $E\cong H_e(T)/\iota$.
	\item $H_+(T)\subset H(T)$ is the set of the so-called ``positive'' half-edges that consists of all regular legs of $T$ and of half-edges in $H(T)\setminus L(T)$ directed away from the root at the vertices where they are attached,
	$H_+(T)\cong E(T)\cup L_{r}(T)$; 
	\item $H_-(T)\subset H(T)$ is the set of the so-called ``negative'' half-edges that consists of all frozen legs of $T$ and of half-edges in $H(T)\setminus L(T)$ directed towards the root at the vertices where they are attached, $H_-(T)\cong E(T)\cup L_{f}(T)$;
	\item $V(T),V_{nr}(T)$ are the sets of vertices and non-root vertices of $T$. 
	\item $v_r\in V(T)$ is the root vertex of $T$; $V(T)=\{v_r(T)\}\sqcup V_{nr}(T)$.
	\item For a $v\in V(T)$, $H(v),H_+(v),H_-(v)$ are all, positive, and negative half-edges attached to $v$, respectively. Obviously, $|H_-(v_r)|=m$ and for any $v\in V_{nr}(T)$ we have $|H_-(v)|=1$.
	\item For a $v\in V(T)$ let $g(v)\in \ZZ_{\geq 0}$ be the genus assigned to $v$. The stability condition means that 
	\[\chi(v)\coloneqq 2g(v)-2+|H(v)|>0.\] 
	The genus condition reads
	\[
	\sum_{v\in V(T)} g(v) = g.
	\]
	\item We say that a vertex or a (half-)edge $x$ is a descendant of a vertex or a (half-)edge $y$ if $y$ is on the unique path connecting $x$ to $v_r$. 
	\item For an $h\in H_+(T)$ let $DL(h)$ be the set of all legs that are descendants to $h$, including $h$ itself. Note that $DL(h)\subseteq L_r(T)$ for any $h\in H_+(T)$ and $DL(l)=\{l\}$ for $l\in L_r(T)$. 
	\item For an $h\in H_+(T)$ let $DH(h)$ be the set of all positive half-edges that are descendants to $h$, \emph{excluding} $h$. For instance, for $l\in L_r(T)$ we have $DH(l) = \emptyset$, and for $h\in H_+(T)\setminus L_r(T)$ we have $DH(h) \supseteq DL(h)$. 
	\item For an $e\in E(T)$ let $DL(e)$ be the set of all legs that are descendants to $e$. Note that $DL(e)\subseteq L_r(T)$ for any $e\in E(T)$. 
	\item For an $v\in E(T)$ let $DL(v)$ be the set of all regular legs that are descendants to $v$. In particular, $DL(v_r) = L_r(T)$. 
	\item For a $v\in V(T)$ let $DV(v)\subset V(T)$ be the subset of all vertices that are descendants of $v$, including $v$ itself. For instance, $DV(v_r) = V(T)$. Let 
	\[
	D\chi(v)\coloneqq \sum_{v'\in DV(v)} \chi(v').
	\]
\end{itemize}

In the pictures it is convenient to arrange the half-edges at each vertex such that the negative half-edges are directed to the left and the positive half-edges are directed to the right. In particular, the root vertex is the leftmost vertex on the pictures. Here is an example of a stable rooted tree in $\SRT_{1,3,2}$ placed on the plane following this convention:
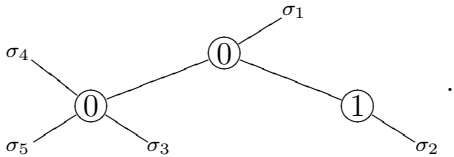
\begin{gather*} 
	\vcenter{\xymatrix@C=15pt@R=5pt{
			& & & & & & \\
			& & & *+[o][F-]{{0}} \ar@{-}[ru]*{{}_{\,\,\sigma_1}} \ar@{-}[rrd] & & & \\ 
			& *+[o][F-]{{0}}\ar@{-}[rru]\ar@{-}[rd]*{{}_{\,\sigma_3}}\ar@{-}[ul]*{{}_{\sigma_4\,\,\,}}\ar@{-}[dl]*{{}_{\sigma_5\,\,\,}} & &  & & *+[o][F-]{{1}} \ar@{-}[rd]*{{}_{\,\,\sigma_2}} & \\
			& & & & & &}}\,.
\end{gather*}

Consider the polynomial ring $Q\coloneqq \QQ[a_1,\dots,a_n]$ and define $a\colon H_+(T) \to Q$, $a\colon E(T)\to Q$, and $a\colon V(T)\to Q$ (abusing notation we use the same symbol $a$ for all these maps) by 
\begin{align*}
	a(\sigma_i)& \coloneqq a_i, &i=1,\dots,n; 
	& & a(h)& \coloneqq \textstyle\sum_{l\in DL(h)} a(l), & h\in H_+(T); \\
	a(e)& \coloneqq \textstyle\sum_{l\in DL(e)} a(l), & e\in E(T); 
	& & a(v)& \coloneqq \textstyle\sum_{l\in DL(v)} a(l), & v\in V(T).
\end{align*} 

\subsection{Vanishing conjecture for more than one frozen leg} \label{sec:vanishing-geq-2}

Let $T\in \SRT_{g,n,m}$. Assign to each $v\in V(T)$ the moduli space of curves $\oM_{g(v),|H(v)|}$, where the first $|H_+(v)|$ marked points correspond to the positive half-edges attached to $v$ and ordered in an arbitrary but fixed way and the the last $|H_-(v)|$ marked points correspond to the negative half-edges attached to $v$, also ordered in some arbitrary but fixed way. Consider the class 
\begin{align*}
& \Omega(v)  \coloneqq a(v)^{1-g(v)}\lambda_{g(v)} \Omega^{[a(v)]}_{g(v),|H(v)|}\Big(a(v),0; -a(h_1),\dots,-a(h_{|H_+(v)|}),\underbrace{0,\ldots,0}_{|H_-(v)|}\Big)
\\ 
& 
\in R^*(\oM_{g(v),|H(v)|})\otimes_{\QQ}Q. 
\end{align*}
The classes $\Omega^{[x]}_{g,n}(r,s;\mu_1,\dots,\mu_n)$ are defined in Section~\ref{sec:Omega:classes}; for now it might be seen just a notation for some class of mixed degree. The only property that we use in this section to state our conjectures and results is the following feature of the class $\Omega(v)$: its component in $R^d$ is a homogeneous polynomial of degree $d$ in $a_1,\dots,a_n$, see Proposition~\ref{prop:homogeneityclassicallimit} in Section~\ref{sec:Omega:classes}.

\begin{definition}
	For each $(g,n.m)$ such that $2g-2+n+m>0$ define the class 
	$$\Omega^m_{g,n}\in R^*(\oM_{g,n+m})\otimes_{\QQ}Q$$
	as
	\begin{equation}
		\Omega^m_{g,n} \coloneqq \sum_{T\in\SRT(g,n,m)} (-1)^{|E(T)|} \biggl(\prod_{e\in E(T)} a(e)\biggr) (b_T)_* \bigotimes_{v \in V(T)} \Omega(v)
	\end{equation}
	Here $(b_T)_*$ is the boundary pushforward map that acts from $\bigotimes_{v \in V(T)} R^*(\oM_{g(v),|H(v)|})\otimes_{\QQ} Q$ to $R^*(\oM_{g,n+m})\otimes_{\QQ}Q$. We call $\Omega(T)$ the summand corresponding to the tree $T$. The class $\Omega^m_{g,n}$ and each $\Omega(T)$ also have the feature that their component in $R^d$ are homogeneous polynomials of degree $d$ in $a_1,\dots,a_n$. 
	

\end{definition}

\begin{conjecture} \label{conj:vanishing-m-less-2} For $g\geq 0$, $n\geq 0$, $m\geq 2$ we have $\deg \Omega^m_{g,n} \leq 2g-2+m$. 
\end{conjecture}

Here the degree $\deg$ can be understood either as the degree in the tautological ring or the degree in the variables $a_1,\dots,a_n$, these two ways to define the degree are equivalent by construction.

This conjecture is supported by the experiments with the \texttt{admcycles} package~\cite{admcycles} for $(g,n,m)=(1,2,2)$ as well as the following two theorems:

\begin{theorem}\label{thm:genus0-mgeq2} For $g=0$, $n\geq 0$, $m\geq 2$ we have $\deg \Omega^m_{0,n} \leq m-2$. 
\end{theorem}
In other words, Conjecture~\ref{conj:vanishing-m-less-2} holds in genus $0$. 

\begin{theorem}\label{thm:anygn1n0-mgeq2} For $g\geq 0$, $n=0,1$, $m\geq 2$ we have $\deg \Omega^m_{g,n} \leq 2g-2+m$. 
\end{theorem}
In other words, Conjecture~\ref{conj:vanishing-m-less-2} holds in any genus for $n=0$ and $n=1$. 

\subsection{An alternative reformulation with level structures} \label{sec:vanishing-geq-2-alt}

We enhance the structure of a stable rooted tree to what we call a degree-labeled stable rooted tree (of genus $g$, with $n$ regular and $m$ frozen legs). To this end we take a stable rooted tree $T\in \SRT_{g,n,m}$ and assign to each $v\in V(T)$ an extra degree label $p(v)\in \ZZ_{\geq 0}$ such that $p(v) \leq 3g(v)-3+|H(v)|$. Denote the set of all degree-labeled stable rooted trees by $\DLSRT_{g,n,m}$.
and consider its elements as pairs $(T,p)$, where $T\in\SRT_{g,n,m}$ and $p$ is the degree label function on $T$.

Our next goal is to assign to a degree-labeled stable rooted tree $(T,p)\in \DLSRT_{g,n,m}$ a coefficient that we call $C_{lvl}(T,p)$. A function $\ell\colon V(T)\to\ZZ_{\geq 0}$ is called an admissible level function if the following conditions are satisfied:
\begin{itemize}
	\item The value of $\ell$ on the root vertex is zero ($\ell(v_r) = 0$).
	\item If $v'\in DV(v)$ and $v'\not=v$, then $\ell(v')>\ell(v)$. 
	\item There are no empty levels, that is, for any $0\leq i \leq \max \ell(V(T))$ the set $\ell^{-1}(i)$ is non-empty. 
	\item For every $0\leq i\leq \max \ell(V(T))-1$ we have inequality 
	\[
	|\{v\in V(T) \,|\, \ell(v)\leq i \}|-1+\sum_{\substack{v\in V(T) \\ \ell(v)\leq i }} p(v) \leq m-2+\sum_{\substack{v\in V(T) \\ \ell(v)\leq i }} 2g(v).
	\]
	In other words, if we let $p(v)$ to be a degree of some class on the moduli space corresponding to the vertex $v$, then the inequality means some bound (\emph{N.B.:} exactly the same bound as in Conjecture~\ref{conj:vanishing-m-less-2}) on the degree of a class obtained by the boundary pushforward map $(b_{T'})_*$ of the tree $T'$ obtained from $T$ by cutting it at level $i$. For instance, if $T$ has more than one vertex, then $p(v_r) \leq 2g(v_r)-2+m$. 
\end{itemize}
Let $\mathcal{L}(T,p)$ denote the set of admissible level functions on $(T,p)\in \DLSRT_{g,n,m}$. The coefficient  $C_{lvl}(T,p)$ is defined as 
\[
C_{lvl}(T,p)\coloneqq \sum_{\ell \in \mathcal{L}(T,p) } (-1)^{\max \ell(V(T))}
\]

Now, we adjust the main steps of the definition of $\Omega(T)$ to the case of degree-labeled stable rooted trees. For a $(T,p)\in\DLSRT_{g,n,m}$, we still assign to each vertex $v\in V(T)$ the moduli space of curves $\oM_{g(v),|H(v)|}$, where the first $|H_+(v)|$ marked points correspond to the positive half-edges attached to $v$ and ordered in an arbitrary but fixed way and the the last $|H_-(v)|$ marked points correspond to the negative half-edges attached to $v$, also ordered in some arbitrary but fixed way. But now we consider the class
\[
(\Omega(v))_{p(v)}\coloneqq \bigg(a(v)^{1-g(v)}\lambda_{g(v)} \Omega^{[a(v)]}_{g(v),|H(v)|}\Big(-a(h_1),\dots,-a(h_{|H_+(v)|}),\underbrace{0,\ldots,0}_{|H_-(v)|}\Big)\bigg)_{p(v)} 
\]
that belongs to $R^{p(v)}(\oM_{g(v),|H(v)|})\otimes_{\QQ}Q$, that is, we consider the homogeneous component of $\Omega(v)$ of degree $p(v)$. It is a homogeneous polynomial of degree $p(v)$ in $a_1,\dots,a_n$, and it can be nonzero only if $p(v)\geq g(v)$ (in principle, one could add the latter inequality to the list of conditions for the degree labels $p$, but we don't do this for a better comparison with~\cite{BS22}, cf.~Remark~\ref{rem:inequ-degree-label} below). 

Assign to each $(T,p)\in \DLSRT_{g,n,m}$ the class $\Omega(T,p)$ in $R^{P(T,p)}(\oM_{g,n+m})\otimes_{\QQ}Q$, where $P(T,p)\coloneqq |E(T)| + \sum_{v\in V(T)} p(v)$, given by 
\[
\Omega(T,p) \coloneqq C_{lvl}(T,p) \biggl(\prod_{e\in E(T)} a(e)\biggr) (b_T)_* \bigotimes_{v \in V(T)} (\Omega(v))_{p(v)}.
\]
Recall that $(b_T)_*$ is the boundary pushforward map from $\bigotimes_{v \in V(T)} R^{p(v)}(\oM_{g(v),|H(v)|})\otimes_{\QQ} Q$ to $R^{P(T,p)}(\oM_{g,n+m})\otimes_{\QQ}Q$. The class $\Omega(T,p)$ is a homogeneous polynomial of degree $P(T,p)$ in $a_1,\dots,a_n$. 

Consider for each $(g,n,m)$ such that $2g-2+n+m>0$ the class ${}^{lvl}\Omega^m_{g,n}  \in R^*(\oM_{g,n+m})\otimes_{\QQ}Q$ defined as
$${}^{lvl}\Omega^m_{g,n} \coloneqq \sum_{(T,p)\in\DLSRT(g,n,m)} \Omega(T,p)$$. 

\begin{conjecture} \label{conj:vanishing-m-less-2-alt} For $g\geq 0$, $n\geq 0$, $m\geq 2$ we have $\deg {}^{lvl}\Omega^m_{g,n} \leq 2g-2+m$. 
\end{conjecture}

As in the case of Conjecture~\ref{conj:vanishing-m-less-2}, the degree $\deg$ can be understood either as the degree in the tautological ring or the degree in the variables $a_1,\dots,a_n$, this two ways to define the degree are equivalent by construction.

\begin{theorem} \label{thm:equivalence-conjectures-m-geq-2} Conjectures~\ref{conj:vanishing-m-less-2} and~\ref{conj:vanishing-m-less-2-alt} are equivalent.
\end{theorem}

So, we can consider Conjecture~\ref{conj:vanishing-m-less-2} as a somehow more streamlined version of Conjecture~\ref{conj:vanishing-m-less-2-alt}. The reason for this more involved reformulation that uses level functions is that this way we can extend our conjecture(s) and theorems to the cases $m=0,1$. Theorem~\ref{thm:equivalence-conjectures-m-geq-2} is a purely combinatorial statement that uses only the degrees of the involved classes; it is proved in Section~\ref{sec:Relation-toBS} below.

Another observation that relates versions of Conjectures~\ref{conj:vanishing-m-less-2} and~\ref{conj:vanishing-m-less-2-alt} for different $m\geq 2$ is the following statement. Let $\pi\colon \oM_{g,n+m} \to \oM_{g,n+m-1}$ be the map that forgets the last marked point.

\begin{theorem} \label{thm:push-forward} For $m>2$ we have 
	\begin{align*}
		\pi_* \Omega^{m}_{g,n} & = (\textstyle\sum_{i=1}^n a_i)\, \Omega^{m-1}_{g,n}; \\
		\pi_* {}^{lvl}\Omega^{m}_{g,n} & = (\textstyle\sum_{i=1}^n a_i)\, {}^{lvl}\Omega^{m-1}_{g,n}.
	\end{align*}
	For $m=2$ we have $\pi_* \Omega^{2}_{g,n} = \pi_* {}^{lvl}\Omega^{2}_{g,n} =0$.
\end{theorem}

This theorem is proved in Section~\ref{sec:Omega-further} along with Theorem~\ref{thm:push-forward-A} stated below.

\subsection{$A$-class conjecture for less than two frozen legs} \label{sec:a-class-leq-1}

First, let us define the $A$-class. 

We consider again the polynomial ring $Q\coloneqq \QQ[a_1,\dots,a_n]$ and define $a\colon H_+(T) \to Q$,  $a\colon H_-(T) \to Q$, and $a\colon E(T)\to Q$ (abusing notation we use the same symbol $a$ for all these maps) by 
\begin{align*}
	a(\sigma_i)& \coloneqq a_i, &i=1,\dots,n; 
	& & a(h)& \coloneqq \textstyle\sum_{l\in DL(h)} a(l), & h\in H_+(T); \\
	a(h)& \coloneqq -\textstyle\sum_{l\in DL(v)} a(l), & h\in H_-(T).
	& & a(e)& \coloneqq \textstyle\sum_{l\in DL(h)} a(l), & e\in E(T).
\end{align*} 
In particular, in this case if two half-edges, $h$ and $h'$, form an edge, then $a(h)+a(h')=0$. Also, for any vertex $v\in V(T)$ we have $\sum_{h\in H(v)} a(h) =0$. Finally, $a(\sigma_{n+1})= -\sum_{i=1}^n a_i$. 

Let $T\in \SRT_{g,n,1}$. Assign to each $v\in V(T)$ the moduli space of curves $\oM_{g(v),|H(v)|}$, where the first $|H_+(v)|$ marked points correspond to the positive half-edges attached to $v$ and ordered in an arbitrary but fixed way and the the last marked point corresponds to the unique negative half-edge attached to $v$. Consider the class
\[
A(v)\coloneqq \lambda_{g(v)}\DR_{g(v)}\big(a(h_1),\dots,a(h_{|H_+(v)|}),a(h_{|H(v)|})\big) \in R^{2g(v)}(\oM_{g(v),|H(v)|})\otimes_{\QQ}Q. 
\]
This class is a homogeneous polynomial of degree $2g(v)$ in $a_1,\dots,a_n$. 

Assign to each $T\in \SRT_{g,n,1}$ the class $A_T$ in $R^*(\oM_{g,n+1})\otimes_{\QQ}Q$ given by 
\[
A(T) \coloneqq  \biggl(\prod_{e\in E(T)} a(e)\biggr)  \biggl(\prod_{v\in V(T)} \frac{\chi(v)}{D\chi(v)}\biggr) (b_T)_* \bigotimes_{v \in V(T)} A(v) \in R^{2g+|E(T)|}(\oM_{g,n+1})\otimes_{\QQ}Q.
\]
Here $(b_T)_*$ is the boundary pushforward map from $\bigotimes_{v \in V(T)} R^{2g(v)}(\oM_{g(v),|H(v)|})\otimes_{\QQ} Q$ to $R^{2g+|E(T)|}(\oM_{g,n+1})\otimes_{\QQ}Q$. The class $A(T)$ s a homogeneous polynomial of degree $2g+|E(T)|$ in $a_1,\dots,a_n$. 

Consider for each $(g,n)$ such that $2g-1+n>0$ the class $A^1_{g,n}\in R^*(\oM_{g,n+1})$ given by $A^1_{g,n} \coloneqq \sum_{T\in \SRT_{g,n,1}} A(T)$. 

\begin{conjecture} \label{conj:m-equal-to-one} For $g\geq 0$, $n\geq 1$ we have $\deg \big({}^{lvl}\Omega^1_{g,n} - A^1_{g,n} \big) \leq 2g-1$. 
\end{conjecture}

This conjecture is supported by the experiments with the \texttt{admcycles} package~\cite{admcycles} for $(g,n)=(1,2)$ and $(2,2)$ as well as the following two theorems:

\begin{theorem} \label{thm:genus0-m1} For $g=0$, $n\geq 2$ we have ${}^{lvl}\Omega^1_{0,n} = A^1_{0,n}$. 
\end{theorem}
In other words, Conjecture~\ref{conj:m-equal-to-one} holds in genus $0$. 

\begin{theorem} \label{thm:anyg-n0n1-m1} For $g\geq 1$, $n=0,1$, we have $\deg \big({}^{lvl}\Omega^1_{g,n} - A^1_{g,n} \big) \leq 2g-1$. 
\end{theorem}
In other words, Conjecture~\ref{conj:m-equal-to-one} holds in any genus for $n=0$ and $n=1$. 

\subsubsection{$A$-class conjecture for no frozen legs} 
Let $\pi\colon \oM_{g,n+1}\to \oM_{g,n}$ be the map that forgets the last marked point. By~\cite[Lemma 2.2]{BGR19}, the following class
\[
A^0_{g,n} \coloneqq \frac{\pi_* A^1_{g,n}}{a_1+\cdots+a_n}
\]
is still a polynomial in $a_1,\dots,a_n$, that is, $A^0_{g,n}\in R^*(\oM_{g,n})\otimes_{\QQ}Q$. 

\begin{conjecture} \label{conj:m-equal-to-zero} For $g\geq 0$, $n\geq 1$ we have $\deg \big({}^{lvl}\Omega^0_{g,n} - A^0_{g,n} \big) \leq 2g-2$. 
\end{conjecture}

This conjecture is supported by the following three theorems:

\begin{theorem}\label{thm:g0-m0} For $g=0$, $n\geq 3$ we have ${}^{lvl}\Omega^0_{0,n} = A^0_{0,n}$. 
\end{theorem}
In other words, Conjecture~\ref{conj:m-equal-to-zero} holds in genus $0$. 

\begin{theorem}\label{thm:anyg-m0} For $g\geq 1$, $n=1$, and $g\geq 2$, $n=0$, we have $\deg \big({}^{lvl}\Omega^0_{g,n} - A^0_{g,n} \big) \leq 2g-2$. 
\end{theorem}
In other words, Conjecture~\ref{conj:m-equal-to-zero} holds in any genus for $n=0$ and $n=1$. 

\begin{theorem}\label{thm:push-forward-A} We have $\pi_* \big({}^{lvl}\Omega^1_{g,n} - A^1_{g,n} \big) = \big(\sum_{i=1}^n a_i\big)\big({}^{lvl}\Omega^0_{g,n} - A^0_{g,n}\big)$.
\end{theorem}
In other words, Conjecture~\ref{conj:m-equal-to-zero} follows from Conjecture~\ref{conj:m-equal-to-one}. In particular Theorems~\ref{thm:g0-m0} and~\ref{thm:anyg-m0} are direct corollaries of Theorem~\ref{thm:push-forward-A} and Theorems~\ref{thm:genus0-m1} and~\ref{thm:anyg-n0n1-m1}.  Theorem~\ref{thm:push-forward-A} is proved in Section~\ref{sec:Omega-further} along with Theorem~\ref{thm:push-forward}.

\subsection{Alternative expressions in terms of $\psi$-classes} \label{sec:Relation-toBS} The structure of the tautological classes used in Conjectures~\ref{conj:vanishing-m-less-2},~\ref{conj:vanishing-m-less-2-alt},~\ref{conj:m-equal-to-one}, and~\ref{conj:m-equal-to-zero} as well as the respective theorems stating their special cases repeat to some extent the conjectures and statements of~\cite{BS22}. Let us explain what are the necessary adjustments to reproduce the expressions discussed in~\cite{BS22}. 

All constructions presented in Sections~\ref{sec:vanishing-geq-2},~\ref{sec:vanishing-geq-2-alt}, and~\ref{sec:a-class-leq-1} are based on the definition of the class $\Omega(v)$ that we assign to a vertex $v$ in a stable rooted tree $T$ or its homogeneous component $(\Omega(v))_{p(v)}$ that we assign to a vertex $v$ in a degree-labeled stable rooted tree $(T,p)$. Now, replace the class $\Omega(v)$ with the class 
\begin{align*}
\Psi(v)& =\Psi_{g(v),|H(v)|}\Big(-a(h_1),\dots,-a(h_{|H_+(v)|}),\underbrace{0,\ldots,0}_{|H_-(v)|}\Big)\bigg)
\\
& \coloneqq \prod_{i=1}^{|H_+(v)|} \frac{1}{1-a(h_i)\psi_i} \in R^*(\oM_{g(v),|H(v)|})\otimes_{\QQ}Q
\end{align*}
(we use the same convention as before, that is, the first $|H_+(v)|$ marked points correspond to the positive half-edges and the last $|H_-(v)|$ correspond to the negative half-edges). With this new assignment one can repeat all steps of the constructions above in order to define $\Psi^m_{g,n}$ for $m\geq 2$ and ${}^{lvl}\Psi^m_{g,n}$ for $m\geq 0$. 

\begin{theorem} \label{thm:relation-to-BS} The vanishing of $(\Psi^m_{g,n})_d$ in degrees $d\geq  2g-1+m$, $m\geq 2$, as well as the vanishing of $({}^{lvl}\Psi^m_{g,n}-\delta^m_0 A^0_{g,n}-\delta^m_1 A^1_{g,n})_d$ in degrees $d\geq  2g-1+m$ is equivalent to~\cite[Conjectures 1, 2, 3]{BS22}.
\end{theorem} 

\begin{remark}\label{rem:inequ-degree-label} Note that using classes $\Omega(v)$ as opposed to $\Psi(v)$ reduces the number of graphs that might contribute non-trivially. Indeed, as we mentioned above, the degree of $\Omega(v)$ is bounded from below by $g(v)$, hence we are only interested in the degree-labeled trees with $p(v)\geq g(v)$ for each $v\in (T,p)$. It is no longer the case for $\Psi(v)$, and it is the reason why we haven't included this inequality in the definition of degree-labeled stable rooted trees. 
\end{remark}

\begin{proof}[Proof of Theorem~\ref{thm:relation-to-BS}] In order to prove this theorem in the degree-leveled case, it is sufficient to identify ${}^{lvl}\Psi^m_{g,n}$ with the generating function $\sum_{\bar d} B^{m}_{g,{\bar d}} \prod_{i=1}^n a_i^{d_i}$, where $B^{m}_{g,{\bar d}}$ is given by~\cite[Theorem 3.10]{BS22}. We recall the latter expression (adapted to the notation that we use in the present paper) and explain the identification. 
	
	Let $(T,p)\in \DLSRT_{g,n,m}$. Introduce an extra function $q\colon H_+(T)\to \ZZ_{\geq 0}$ such that for each $v\in V(T)$ we have $p(v) = \sum_{h\in H_+(v)} q(h)$. Denote $\mathcal{Q}(T,p)$ the set of such functions. Then $B^m_{g,\bar d}$, $\bar d = (d_1,\dots,d_n)$, is defined as 
	\begin{align*}
		B^m_{g,\bar d} & \coloneqq  \sum_{(T,p)\in \DLSRT_{g,n,m}} C_{lvl}(T,p) \sum_{q\in \mathcal{Q}(T,p)} 
		\Bigg((b_T)_* \bigotimes_{v \in V(T)} \prod_{i=1}^{|H_+(v)|} \psi_i^{q(h_i)} \Bigg) \times
		\\ & \qquad \qquad \qquad \qquad \qquad 
		\frac{\prod_{h\in H_+(T)} \big(\sum_{l\in DL(h)}(d(l)+1) - \sum_{h'\in DH(h)} (q(h')+1) \big)_{q(h)+1}}{\prod_{i=1}^n (d_i+1)!}.
	\end{align*}
	Here $d(l)$ is defined as $d(\sigma_i)=d_i$ and $(s)_{t}=s(s-1)\cdots (s-t+1)$ denotes the Pochhammer symbol and it is assumed that $\sum_{h\in H_+(T)} q(h) + |E(T)| = \sum_{i=1}^n d_i$. 
	
	In order to prove that $\sum_{\bar d} B^m_{g,\bar d} \prod_{i=1}^n a_i^{d_i} = {}^{lvl}\Psi^m_{g,n}$, it is sufficient to show that the expressions match for each $(T,p)\in \DLSRT_{g,n,m}$ separately. This is done by applying the following identity at each vertex of $T$:
	\[
	\sum_{\substack{q\colon H_+(v)\to \ZZ_{\geq 0} \\ \sum_{h\in H_+(v)} q(h) = p(v)}} \prod_{i=1}^{|H_+(v)|} \Bigg(\psi_i^{q(h_i)} a(h_i)^{c_i}  \frac{(c_i+1)_{q(h_i)+1}}{(c_i+1)!} \Bigg) = (\Psi(v))_p \cdot a(v) \cdot \frac{a(v)^{\sum_{i=1}^n c_i-p-1}}{\big(\sum_{i=1}^n c_i-p\big)!}
	\]
	for suitable $c_i\geq q(h_i)$, $i=1,\dots,|H_+(v)|$.
	
	In the case of the usual $\Psi^m_{g,n}$, it has to be compared with the generating function $\sum_{\bar d} \tilde B^{m}_{g,{\bar d}} \prod_{i=1}^n a_i^{d_i}$, where $\tilde B^{m}_{g,{\bar d}}$ is given by~\cite[Equation 3.2]{BS22}. The argument is then completely parallel to the one presented above in the case of degree-leveled expressions. 
\end{proof}

Theorem~\ref{thm:relation-to-BS} provides a compact reformulation of the classes introduced and considered in~\cite{BGR19,BHS22,BS22}. But, most importantly, in the context of this paper it allows to automatically transfer the purely combinatorial equivalences proved in~\cite{BS22} to the case of classes $\Omega^{m}_{g,n}$ and ${}^{lvl}\Omega^{m}_{g,n}$ considered in the present paper. For instance, we can immediately prove Theorem~\ref{thm:equivalence-conjectures-m-geq-2}.

\begin{proof}[Proof of Theorem~\ref{thm:equivalence-conjectures-m-geq-2}] It is a direct corollary of Theorem~\ref{thm:relation-to-BS} and the proofs of~\cite[Theorems 3.4 and 3.10]{BS22} --- the latter proofs provide a combinatorially described triangular change that relates $\{\Psi^{m}_{g,n}\}$ and  $\{{}^{lvl}\Psi^{m}_{g,n}\}$, and since nothing but the vanishing dimensions of the classes are used, we can replace them by $\{\Omega^{m}_{g,n}\}$ and  $\{{}^{lvl}\Omega^{m}_{g,n}\}$ in that argument. 
\end{proof}

\section{Background on $\Omega$-classes}
\label{sec:Omega:classes}

\subsection{Definition and formula}
In \cite{Mum83}, Mumford derived a formula for the Chern character of the Hodge bundle on the moduli space of curves $\overline{\mathcal{M}}_{g,n}$ in terms of tautological classes and Bernoulli numbers. 
%
%
A generalisation of Mumford's formula was computed in \cite{Chi08}. The moduli space $\overline{\mathcal{M}}_{g,n}$ is substituted by the proper moduli stack $\overline{\mathcal{M}}_{g;a}^{r,s}$ of $r$-th roots of the line bundle
\begin{equation*}
	\omega_{\log}^{\otimes s}\biggl(-\sum_{i=1}^n a_i p_i \biggr),
 \end{equation*}
where $\omega_{\log} = \omega(\sum_i p_i)$ is the log-canonical bundle, $r$ and $s$ are integers with $r$ positive, and $a_1,\dots,a_n$ are integers 
satisfying the modular constraint
\begin{equation*}
	a_1 + a_2 + \cdots + a_n \equiv (2g-2+n)s \pmod{r}.
\end{equation*}
This condition guarantees the existence of a line bundle whose $r$-th tensor power is isomorphic to $\omega_{\log}^{\otimes s}(-\sum_i a_i p_i)$. Let $\pi \colon \overline{\mathcal{C}}_{g;a}^{r,s} \to \overline{\mathcal{M}}_{g;a}^{r,s}$ be the universal curve, and $\mathcal{L} \to \overline{\mathcal C}_{g;a}^{r,s}$ the universal $r$-th root. In complete analogy with the case of moduli spaces of stable curves, one can define $\psi$-classes and $\kappa$-classes. There is moreover a natural forgetful morphism
\begin{equation*}
	\epsilon \colon
	\overline{\mathcal{M}}^{r,s}_{g;a}
	\longrightarrow
	\overline{\mathcal{M}}_{g,n}
\end{equation*}
which forgets the choice of the line bundle.
It can be turned into an unramified covering in the orbifold sense of degree $2g - 1$ by slightly modifying the structure of $\overline{\mathcal{M}}_{g,n}$, introducing an extra $\ZZ_r$ stabilizer for each node of each stable curve (see~\cite{JPPZ}).


Let $B_m(x)$ denote the $m$-th Bernoulli polynomial, that is the polynomial defined by the generating series
\begin{equation*}
	\frac{te^{tx}}{ e^t - 1} = \sum_{m=0}^{\infty} B_{m}(x)\frac{t^m}{m!}.
\end{equation*}
The evaluations $B_m(0) = (-1)^m B_m(1) = B_m$ recover the usual Bernoulli numbers. Chiodo's formula provides an explicit formula for the Chern characters of the derived pushforward of the universal $r$-th root $\ch_m(r,s;
\vec a) = \ch_m(R^{\bullet} \pi_{\ast}{\mathcal L})$.

\begin{theorem}[\cite{Chi08}]
	The Chern characters $\ch_m(r,s;\vec a)$ of the derived pushforward of the universal $r$-th root have the following explicit expression in terms of $\psi$-classes, $\kappa$-classes, and boundary divisors:
	\begin{equation} \label{eqn:Chiodo:formula}
		\ch_m(r,s;\vec a)
		=
		\frac{B_{m+1}(\tfrac{s}{r})}{(m+1)!} \kappa_m
		-
		\sum_{i=1}^n \frac{B_{m+1}(\tfrac{a_i}{r})}{(m+1)!} \psi_i^m
		+
		\frac{r}{2} \sum_{a=0}^{r-1} \frac{B_{m+1}(\tfrac{a}{r})}{(m+1)!} \, j_{a,\ast} \frac{(\psi')^m - (-\psi'')^m}{\psi' + \psi''}. 
	\end{equation}
	Here $j_a$ is the boundary morphism that represents the boundary divisor with multiplicity index $a$ at one of the two branches of the corresponding node, and $\psi',\psi''$ are the $\psi$-classes at the two branches of the node.
\end{theorem}

We can then consider the family of Chern classes pushed forward to the moduli spaces of stable curves
\begin{equation*}
	\Omega_{g,n}^{[x]}(r,s;a_1,\dots,a_n)
	\coloneqq 
	\epsilon_{\ast}  \exp{\Biggl(
		\sum_{m=1}^\infty (-1)^m x^{m} (m-1)! \, \ch_m(r,s;\vec{a})
	\Biggr)}
	\in
	R^*(\overline{\mathcal{M}}_{g,n}).
\end{equation*}
 We will omit the variable $x$ when $x = 1$. Notice that we recover Mumford's formula for the Hodge class when $r = s = 1$, $x=1$, and $a = (1,\dots,1)$.  The indices $a_1,\dots,a_n$ are often referred to as ``primary fields''.
 
 \begin{corollary}[\cite{JPPZ}] \label{cor:JPPZExp}
The class 
 $\Omega_{g,n}^{[x]}(r,s;a_1,\dots,a_n)$ is equal to
\begin{align}  \label{eqn:JPPZExp}
\sum_{\Gamma\in \mathsf{G}_{g,n}} 
\sum_{w\in \mathsf{W}_{\Gamma,r,s}}
\frac{r^{2g-1-h^1(\Gamma)}}{|\Aut(\Gamma)| }
\;
\xi_{\Gamma *}\Bigg[ &  \prod_{v \in V(\Gamma)} e^{-\sum\limits_{m = 1} (-1)^{m-1} x^m \frac{B_{m+1}(s/r)}{m(m+1)}\kappa_m(v)} \; \times 
\\ \notag
& 
\, \, \, \prod_{i=1}^n e^{\sum\limits_{m = 1}(-1)^{m-1} x^m \frac{B_{m+1}(a_i/r)}{m(m+1)} \psi^m_{i}} \; \times
\\ \notag
& \!  \prod_{\substack{e\in E(\Gamma) \\ e = (h,h')}}
\frac{1-e^{\sum\limits_{m \geq 1} (-1)^{m-1} x^m \frac{B_{m+1}(w(h)/r)}{m(m+1)} [(\psi_h)^m-(-\psi_{h'})^m]}}{\psi_h + \psi_{h'}} \Bigg]\, .
\end{align} 
Here $\mathsf{G}_{g,n}$ is the finite set of stable graphs of genus $g$ with $n$ legs. The set $\mathsf{W}_{\Gamma,r,s}$ is the finite set of decorations of all half-edges (including the legs). The leg $i$ is decorated with $a_i$, and any other half-edge $h$ with an integer $w(h) \in \{0, \dots, r-1\}$ (a so-called ``weight'') in such a way that decorations of half-edges $h,h'$ of the same edge $e \in E(\Gamma)$, $e=(h,h')$ sum up to $r$ and locally on each vertex $v \in V(\Gamma))$ of genus $g(v)$ and index $n(v)$  the sum of all decorations is congruent to $(2g(v) - 2 + n(v))s$ modulo $r$. 
\end{corollary}

\subsection{Basic properties} \label{sec:Omega-basic} This section recalls several properties of the $\Omega$-classes collected in~\cite{gln}. Fix $g,n \geq 0$ integers such that $2g - 2 + n > 0$. Let $r$ and $s$ be integers with $r$ positive. Let $a_1, \ldots, a_n$ be integers satisfying the modular constraint $a_1+a_2+\cdots+a_n \equiv (2g-2+n)s \pmod{r}$. The $\Omega$-classes satisfy the following properties:

\begin{lemma}[\cite{gln}]  \label{prop:shift-ai}We have:
	\begin{align*}
		\Ch^{[x]}_{g,n}(r,s;a_1, \dots, a_i + r, \dots, a_n)
		& =
		\Ch^{[x]}_{g,n}(r,s;a_1, \dots, a_n) \cdot\left( 1 + x\frac{a_i}{r}\psi_i\right) \\
		\Ch^{[x]}(r,0;a_1, \dots, a_n) & =
	\Ch^{[x]}(r,r;a_1, \dots, a_n) \\
	\Ch^{[x]}(r,s; a_1, \dots, 0, \dots, a_n)
	& =
	\Ch^{[x]}(r,s;a_1, \dots, r, \dots, a_n)
	\end{align*}
\end{lemma}

\begin{lemma}[\cite{gln,lpsz}] \label{lem:pull-back-range} Assume in addition that $0\leq a_1,\dots,a_n\leq r$. Then 
	\begin{equation*}
	\Ch^{[x]}_{g,n}(r,s;a_1, \dots, a_n, s)
	=
	\pi^{\ast}\Ch^{[x]}_{g,n}(r,s;a_1, \dots, a_n),
\end{equation*}
where $\pi\colon \oM_{g,n+1}\to \oM_{g,n}$ forgets the last marked point. 
\end{lemma}

\subsection{Further properties} \label{sec:Omega-further} Here we collect the properties important for the statements in Section~\ref{sec:results}. 

\begin{lemma}
\label{lem:poly_a_ct} 
Let $x$ be a formal variable, let $a_1, \dots, a_n$ be integer numbers and let $r$ to be a sufficiently large natural number. We have that the expression
$$
\Omega_{g,n}^{[x]}(r,0; a_1, \dots, a_n)\Big{|}_{\oM_{g,n}^{ct}} \in R^{*}(\oM^{ct}_{g,n})
$$
is a (non-homogeneous) polynomial in the variables $a_i$.
 \begin{proof}
Let us analyse the contributions appearing in the explicit formula \eqref{cor:JPPZExp}. The restriction to compact type implies that the class above reduces from a sum over stable graphs to a sum over stable trees. This means that the amount of terms given by the cardinality of the two sums does not depend on the $a_i$'s: the first one only depends on the topology, the second one is trivial since all decorations at the half-edges in a tree are uniquely determined by the decorations at the leaves (which are equal to the $a_i$, $i=1,\dots,n$).

The contributions of the vertices do not depend on the $a_i$, the contributions of the leaves are manifestly polynomial in the $a_i$, and so are the contributions of the separable nodes, as the weights $w(h)$ are uniquely determined as linear functions of the decorations $a_i$ on the leaves. This concludes the proof of the lemma.\end{proof}
\end{lemma}

\begin{proposition}
\label{prop:homogeneityclassicallimit} 
Let $a_1, \dots, a_n$ be positive integers and let $a=a_1+\cdots+a_n$. Then the cohomological degree $k$ expression 
$$
\left( a^{1-g} \lambda_g \Omega_{g,n}^{[a]}(a,0; -a_1, \dots, -a_n)\right)_k \in R^k(\oM_{g,n})
$$
is an homogeneous polynomial in the $a_i$ of polynomial degree equal to $k$.
\end{proposition}

\begin{proof} Consider the class $a^{1-g}\lambda_g \Omega_{g,n}^{[x]}(a,0; -a_1, \dots, -a_n)$ with so far no relations between $x, a$, and $a_1,\dots, a_n$.
Let us first notice that the expression is trivially polynomial in the variable $x$ keeping track of the cohomological degree of the $\Omega$-class, hence setting $x$ to $a$ only provides further polynomial dependency in $a$. Recall that the class $\lambda_g$ vanishes on the complement of the compact type, hence by Lemma \ref{lem:poly_a_ct} the dependency on the individual $a_i$'s is polynomial. Furthermore, let us notice that from the analysis of \eqref{eqn:JPPZExp} and the proof of Lemma~\ref{lem:poly_a_ct} $\Omega_{g,n}^{[a]}(a,0; -a_1, \dots, -a_n)$ restricted to compact type is a Laurent polynomial in $a$. This expression as a whole has homogeneous degree $2g - 1 + d$ in $a,a_1,\dots,a_n$ for the cohomological degree $d$ part of $\Omega_{g,n}^{[a]}(a,0; -a_1, \dots, -a_n)$. After multiplication with $a^{1-g}\lambda_g$, the degree $k$ of the whole expression $a^{1-g}\lambda_g \Omega_{g,n}^{[a]}(a,0; -a_1, \dots, -a_n)$ is hence a homogeneous Laurent polynomial in $a$ and $a_i$'s of polynomial degree $2g - 1 + (k - g) + (1 - g) = k$ (more precisely, it is a Laurent polynomial in $a$ and an ordinary polynomial in $a_1,\dots,a_n$).

In the case $n=0$ all weights $w(h)=0$ and we see that the resulting expression for $a^{1-g}\lambda_g \Omega_{g,n}^{[a]}(a,0;)$ is manifestly a polynomial in $a$ of the prescribed degree. So we assume from now one that $n\geq 1$.  In this case we are left with establishing the polynomial dependency of the expression on the variable $a$ once we further restrict $a = \sum_{i=1}^n a_i$. For this purpose we are going to employ a result in \cite{Fan21}, which states that 
$$
r^{d-g+1} \left(\Omega_{g,n}(r,0;b_1, \dots, b_n)\right)_d
$$
is a polynomial in $r$ for sufficiently large $r$ whenever $b_1 + \dots + b_n = 0$. Sufficiently large in this case means that no subset of the $b_i$ can sum up to $r$ (this is the condition that guarantees the second sum in \eqref{eqn:JPPZExp} to be independent on $r$). In order to apply this result, we can apply Lemma~\ref{prop:shift-ai} to the first primary field $a_1$. In this way we can restate the class as
\begin{align*}
& a^{1-g} \left( \Omega_{g,n}^{[a]}(a,0; -a_1, \dots, -a_n)\right)_k =  \left( \frac{ a^{1-g} \Omega_{g,n}^{[a]}(a,0; a-a_1, a_2, \dots, -a_n)}{(1 - a_1\psi_1 )} \right)_k
\\
&= \sum_{d=0}^k 
 \bigg(r^{d-g+1} \left(\Omega_{g,n}(r,0; \textstyle \sum_{j=2}^n a_j, \dots, -a_n)\right)_{d} \bigg)\bigg|_{r=\sum_{i=1}^n a_i} \left(\frac{1}{1 - a_1 \psi_1 }\right)_{k-d}
\end{align*}
The second factor of each summand in the latter sum is manifestly polynomial in $a_1$, and the first factor is polynomial in $r$ by \cite{Fan21}. Note that $r=\sum_{i=1}^n a_i$ is here indeed sufficiently large as it cannot interact with subsets of the primary fields. This concludes the proof of the proposition for $n\ge 1$.
\end{proof}

\begin{remark} As an alternative argument, one can reduce the polynomiality of 	the class $a^{1-g}  \Omega_{g,n}^{[a]}(a,0; -a_1, \dots, -a_n)$ for $n\geq 1$ to a vanishing statement in~\cite{CJ} via the same trick with the shift of the first primary field applied twice .
\end{remark}

Another statement that we need is the following:

\begin{proposition} \label{prop:pull-back} For any $m\geq 1$, 
	\begin{align*}
		a^{1-g}\lambda_{g} \Omega^{[a]}_{g,n+m}\big(a,0; -a_1,\dots,-a_n,\underbrace{0,\ldots,0}_{m}\big)
	= \frac{a^{1-g}\lambda_{g} \pi_m^* \Omega^{[a]}_{g,n}(a,0; a-a_1,\dots,a-a_n)}{\prod_{i=1}^n (1-a_i\psi_i)},
	\end{align*}
	where $\pi_m\colon \oM_{g,n+m} \to \oM_{g,n}$ is the map forgetting the last $m$ marked points.
\end{proposition}

\begin{proof} First, we apply Lemma~\ref{prop:shift-ai}:
\begin{align*}
		& a^{1-g}\lambda_{g} \Omega^{[a]}_{g,n+m}\big(a,0; -a_1,\dots,-a_n,\underbrace{0,\ldots,0}_{m}\big) \\
		& = a^{1-g}\lambda_{g}  \Omega^{[a]}_{g,n}\big(a,0; a-a_1,\dots,a-a_n,\underbrace{0,\ldots,0}_{m}\big) \frac 1{\prod_{i=1}^n (1-a_i\psi_i)}. 
\end{align*}
This brings us to the range where we can apply Lemma~\ref{lem:pull-back-range} $m$ times to the first factor, which completes the proof.
\end{proof}

With Proposition~\ref{prop:pull-back} at hand, we are able to prove the push-forward statements.

\begin{proof}[Proofs of Theorems~\ref{thm:push-forward} and~\ref{thm:push-forward-A}] Note that for $\big({\prod_{i=1}^n (1-a_i\psi_i)}\big)^{-1} \in R^*(\oM_{g,n+1})\otimes_{\QQ} Q$ and $\pi\colon \oM_{g,n+1}\to \oM_{g,n}$ we have 
\begin{align} \label{eq:push-forward-psi}
	\pi_* \frac 1{\prod_{i=1}^n (1-a_i\psi_i)} = \frac {\sum_{i=1}^n a_i}{\prod_{i=1}^n (1-a_i\psi_i)} \in R^*(\oM_{g,n})\otimes_{\QQ} Q.
\end{align}

In the setup of Theorem~\ref{thm:push-forward} for $m\geq 3$ and $m=1$ the push-forward that forgets the last marked point (corresponding to the last frozen leg in the stable rooted trees) does not change the structure of the graph, or, in the case of the class ${}^{lvl}\Omega_{g,n}^m$, the level structure and implied coefficients $C^{lvl}$. It only acts on the class $\Omega(v_r)$ assigned to the root vertex and, in the degree-labeled case, decreases the value of $p(v_r)$ by $1$. At the root vertex we can combine Equation~\eqref{eq:push-forward-psi} and Proposition~\ref{prop:pull-back} via the projection formula to see that the decoration of the root vertex in the resulting graph is exactly the expected one, with an extra factor of $\sum_{i=1}^n a_i$.

The case of $m=2$ is a bit special, since under the push-forward we have to take into account the graphs with the root vertex $v_r$ of genus $0$, with $|H_+(v_r)|=1$. This root vertex is decorated by $\Omega(v_r) = 1$ and has to be contracted under the pushforward. The new root vertex $v_r'$ is then then one that was a direct descendant of the root vertex before the push-forward. We gain an extra factor $\sum_{i=1}^n a_i$ for the edge connecting $v_r$ and $v_r'$. Under the pushforward we also gain an extra sign $(-1)$, either as an extra sign for the edge connecting $v_r$ and $v_r'$ in the case of $\Ch^2_{g,n}$, or as a consequence of the fact that with removing $v_r$ we suppress exactly one level in the computation of $C_{lvl}(T,p)$ for ${}^{lvl}\Ch^2_{g,n}$.

The graphs with the root vertex $v_r$ of genus $0$, with $|H_+(v_r)|=1$ are in one-to-one correspondence with the graphs where $2g(v_r)-2+|H_+(v_r)|>0$  both in the degree-labeled and the non-degree-labeled cases. The correspondence is given by contracting the edge connecting $v_r$ and $v_r'$ (in the leveled case note that $p(v_r)=0$ one $m=2$ and $g(v_r)=0$). The pairs of graphs such correspondence have identical push-forwards but with opposite signs, thus the resulting class is equal to zero.
\end{proof}

\subsection{Genus 0} In this Section we analyze the $\Ch$-class in genus $0$ and prove Theorems~\ref{thm:genus0-mgeq2} and~\ref{thm:genus0-m1}. 

\begin{lemma}
\label{lem:g0vanishing}
For genus zero, we have 
$$
\left(a^{1} \Omega_{0,n}^{[a]}(a, 0; a - a_1, \dots, a - a_n)\right)_k=
\begin{cases}
0 \quad \text{ if } k > 0, 
\\
1 \quad \text{ if } k = 0.
\end{cases}
$$
\end{lemma}

 \begin{proof}
In cohomological degree zero and any genus, the $\Omega$-class is well known to be equal to the fundamental class times $a^{2g-1}$, hence for $k=0$ the statement is trivial. In fact the multiplication by powers of $a$ in front of the $\Omega$-class serve the purpose of normalisation. 
For positive cohomological degree, by Lemma~\ref{prop:shift-ai} we have
\begin{equation}\label{eq:swap_s}
\Omega_{0,n}(a,0; a-a_1, \dots, a -a_n) = \Omega_{0,n}(a,a; a - a_1, \dots,a - a_n).
\end{equation}
Moreover, by \cite[Proposition 4.4]{JKV1},  
\begin{equation*}
\Omega_{0,n}(r,0; b_1, \dots, b_n), \qquad \qquad b_i > 0
\end{equation*}
is the properly normalized by a multiplicative constant push-forward from the space of $r$-spin structures of an actual total Chern class of a vector bundle (as opposed to a virtual bundle). This occurs because the universal $r$-th root does not have any global section as long as genus zero is considered and all primary fields are positive. 
All $b_i = a - a_i$ are positive as each $a_i >0$ and as for stability we need $n \geq 3$. We can then perform the Riemann-Roch computation of the rank (and we do so on the RHS of \ref{eq:swap_s}, that is, with $s=a$), which in general leads to:
\begin{equation*}
\frac{(2g - 2 + n)s - \sum_i b_i}{r} - g + 1 = h^0 - h^1
\end{equation*}
Imposing $g=0$, $h^0 = 0$, and substituting $r=s=a$ and $a_i = a - a_i$ we find that the rank of the corresponding vector bundle is equal to:
\begin{equation*}
h^1 = \frac{\sum_i (a - a_i) - (n-2)a }{a} - 1 = 0.
\end{equation*}
This concludes the proof of the lemma.
\end{proof}

As an immediate corollary of this Lemma and Proposition~\ref{prop:pull-back}, we have:

\begin{corollary} For any vertex $v$ in a stable rooted tree such that $g(v)=0$, we have $\Omega(v)=\Psi(v)$. In particular, $\Omega^{m}_{0,n} = \Psi^m_{0,n}$, $m\geq 2$, and ${}^{lvl}\Omega^{m}_{0,n} = {}^{lvl}\Psi^m_{0,n}$, $m\geq 0$. 
\end{corollary}

This corollary means that Theorems~\ref{thm:genus0-mgeq2},~\ref{thm:genus0-m1}, and~\ref{thm:g0-m0} are equivalent to the statements proved in~\cite[Theorem 2.3]{BS22}, cf. Theorem~2.3 in \emph{op.~cit.}

A straightforward generalisation of the argument above via Riemann-Roch gives the following proposition, which we do not need but nevertheless report on for completeness.

\begin{proposition}
\label{prop:RRs0}
Let $a$ be a positive integer and let $b_1, \dots, b_n$ be integers such that their sum is divisible by $a$ and let $B$ be the quotient. Then we have 
\begin{align*}
\left(\Omega_{g,n}(a, 0; b_1, \dots, b_n)\right)_k &= 0 \text{ whenever } k \geq g + B \quad & \text{ if } B > 0
. 
\end{align*}
\end{proposition}

\subsection{The $n=1$ case} We have the following 

\begin{proposition} \label{prop:reducltion-to-lambda} Let $n=1$. Then for any $m\geq 1$, 
	\begin{align*}
		a^{1-g}\lambda_{g} \Omega^{[a]}_{g,1+m}\big(a,0; -a,\underbrace{0,\ldots,0}_{m}\big)
		= \frac{a^{g}\lambda_{g} \Lambda_g^{[a]} }{(1-a\psi_1)},
	\end{align*}
	where $\Lambda_g^{[a]}$ denotes $\sum_{i=0}^g (-1)^i\lambda_i a^i$.
\end{proposition}

\begin{proof} Proposition~\ref{prop:pull-back} implies that it is sufficient to prove that  
$$
	\lambda_g\Omega^{[a]}_{g,1}(a,0; 0)= a^{2g-1} \lambda_g \Lambda_g^{[a]}.
$$ 
To this end, we just match the contributions of the graphs in Equation~\ref{eqn:JPPZExp} for the left hand side and the corresponding expression stemming from Mumford's formula for the right hand side for this equation. 
\begin{enumerate}
	\item The presence of $\lambda_g$ forces the set of stable graphs to restrict to the set of stable trees. Therefore the prefactor $a^{2g-1-h^1(\Gamma)}$ specializes to $a^{2g-1}$ and can be factored outside. 
	\item For stable trees, via the local vertex condition all decorations are uniquely determined by the decorations of the legs. In this case, the only leg is decorated by zero, forcing all other decorations to be equal to zero. Therefore the Bernoulli polynomials present in the edges contributions become Bernoulli numbers independent on $a$.
	\item The Bernoulli polynomials paired with $\psi$ or $\kappa$ classes get evaluated at zero and therefore also specialize to Bernoulli numbers independent on $\mu$.
\end{enumerate}
Thus the formula specializes to the Mumfords formula for $\Lambda_g^{[a]}$ multiplied by $\lambda_g$ (thus we also have to select only stable trees among all graphs) and an extra factor $a^{2g-1}$.
\end{proof}

Thus, in order to prove Theorems~\ref{thm:anygn1n0-mgeq2} and~\ref{thm:anyg-n0n1-m1} for $n=1$, we have to make computations only with the $\lambda$-classes. To this end, we generalize some known relations for $\lambda$-classes in Section~\ref{sec:Arcara:Sato} below.

\begin{remark} In the light of Proposition~\ref{prop:reducltion-to-lambda} one might wonder why we pose our conjectures in Section~\ref{sec:results} in terms of the $\Ch$-classes. A plausible alternative might be, for instance, to replace $\Omega(v)$ with  $a(v)^{g(v)}\lambda_{g(v)} \Lambda_{g(v)}^{[a(v)]} \Psi(v)$; in all proven cases there is no difference between these classes. However, the very first case where these classes can potentially give different result, $g=1$, $m=n=2$, we indeed have the vanishing prescribed by Conjecture~\ref{conj:vanishing-m-less-2} for the $\Ch$-classes, and no vanishing for this alternative, as per experiments with \texttt{admcycles}.
\end{remark}

\subsection{The $n=0$ case} \label{sec:n0-case}

\begin{proposition} \label{prop:reducltion-to-lambdan0} Let $n=0$. Then for any $m\geq 1$ we have
	\begin{align*}
		a^{1-g}\lambda_{g} \Omega^{[a]}_{g,m}\big(a,0; \underbrace{0,\ldots,0}_{m}\big)
		= a^{g}\lambda_{g} \Lambda_g^{[a]}.
	\end{align*}
\end{proposition}

\begin{proof} The proof is parallel to the one of Proposition~\ref{prop:reducltion-to-lambda}.
\end{proof}

Employing that $\lambda_g^2 = 0$, we immediately have the following corollary.

\begin{corollary} For $m \geq 1$ and $n=0$,
$
		\deg\left(\lambda_{g} \Omega^{[a]}_{g,m}\big(a,0; 0,\ldots,0\big)\right) \leq 2g - 1.
$
\end{corollary}

This concludes the proof and refines the bound of Theorems~\ref{thm:anygn1n0-mgeq2} and~\ref{thm:anyg-n0n1-m1} for the $n=0$ case, as the $\Omega$-classes involved for $n=0$ degenerate from sums of (degree-labeled) stable trees to a single rooted vertex and $A^1_{g,0}=\lambda_g^2=0$.

\section{Arcara-Sato relations and generalizations}
\label{sec:Arcara:Sato}

\subsection{Basis relation in the tautological ring}
Let $\Lambda_g\coloneqq \sum_{i=0}^g (-1)^i \lambda_i$ (we omit index $[x]$ in the notation $\Lambda_g^{[x]}$ when we set $x=1$).

 Let $\gamma_1\diamond\gamma_2$ be the operation of concatenation of two classes or decorated stable dual graphs, $\gamma_1$ in $\oM_{g_1,2}$, $g_1\geq 1$, and $\gamma_2$ in $\oM_{g_2,1+m}$, $g_2\geq 1$ for $m=0,1$ and $g_2\geq 0$ for $m\geq 2$, that forms an edge from the second leaf of $\gamma_1$ and the first leave of $\gamma_2$. 

\begin{lemma}
For any $r\geq 0$ we have following relation in $\overline{\mathcal{M}}_{g,2}$:
\begin{align}
	\label{eq: AS m=2} 0 = \ & - \Big(\frac{\Lambda_g}{1-\psi_1}\Big|_{\oM_{g,2}}\Big)_{g+r}
	+ (-1)^{r} \Big(\frac{\Lambda_g}{1-\psi_2}\Big|_{\oM_{g,2}}\Big)_{g+r}
	\\ \notag
	& + \sum_{\substack{g_1+g_2=g\\ g_1,g_2\geq 1 \\ a_1+a_2 =g-1+r}} (-1)^{g_1+a_1} 
	\Big(\frac{\Lambda_{g_1}}{1-\psi'}\Big|_{\oM_{g_1,2}}\Big)_{a_1}
	\diamond 
	\Big(\frac{\Lambda_{g_2}}{1-\psi''}\Big|_{\oM_{g_2,2}}\Big)_{a_2}. 
\end{align}	
Moreover, for any $m\geq 2$ and $r\geq 0$ we have the following relation in $\overline{\mathcal{M}}_{g,1+m}$:
\begin{align}
	\label{eq: AS any m} 0 = \ & - \Big(\frac{\Lambda_g}{1-\psi_1}\Big|_{\oM_{g,1+m}}\Big)_{g+m-1+r}
	\\ \notag
	& + \sum_{\substack{g_1+g_2=g\\ g_1\geq 1, g_2\geq 0\geq \\ a_1+a_2 =g+m-2+r}} (-1)^{g_1+a_1} 
	\Big(\frac{\Lambda_{g_1}}{1-\psi'}\Big|_{\oM_{g_1,2}}\Big)_{a_1}
	\diamond 
	\Big(\frac{\Lambda_{g_2}}{1-\psi''}\Big|_{\oM_{g_2,1+m}}\Big)_{a_2}. 
\end{align}	
In both formulas $\psi'$ and $\psi''$ denote the $\psi$-classes at the two branches of the node formed by the operation $\diamond$. 
\end{lemma}

\begin{proof} 
Formulas~(\ref{eq: AS m=2}) and (\ref{eq: AS any m}) are obtained by a slight modification of the argument of Arcara and Sato in \cite{AS09}. In this context, we will only point out the necessary modifications in their proof to derive a valid proof of (\ref{eq: AS m=2}) and (\ref{eq: AS any m}), while referring to their work for notations and details.

In this \cite{AS09}, Arcara and Sato first define a $\mathbb{C}^*$-action on the moduli space of degree $1$ maps from genus $g$ curves with $3$ points to $\mathbb{P}^{1}$ denoted by $\overline{\mathcal{M}}_{g,3}\left(\mathbb{P}^{1},1\right)$. Then, they use the virtual localization formula \cite{GP97} on $\overline{\mathcal{M}}_{g,3}\left(\mathbb{P}^{1},1\right)$, imposing in each fixed locus that the point $1$ is sent to $0\in\mathbb{P}^{1}$, while the points $2$ and $3$ are sent to $\infty\in\mathbb{P}^{1}$. Then, they extract the coefficient of $t^{-4}$ in this formula, where $t$ is the generator of the equivariant cohomology of the point. This coefficient vanishes and this vanishing gives their formula after forgetting the marked points $2$ and $3$. 

In our case, we consider the moduli space $\overline{\mathcal{M}}_{g,1+m}\left(\mathbb{P}^{1},1\right)$, for $m\geq2$, that is with $m+1$ marked points instead of $3$, and we impose in each fixed locus the first marked point to be sent to $0$, while the $m$ other marked points are sent to $\infty$. Thus, their localization formula is replaced by
\[
\sum_{\tilde{F}_{i}}\left(\tilde{f}\vert_{\tilde{F}_{i}}\right)_{*}\frac{\left[1\right]^{{\rm vir}}}{e_{\mathbb{C}^{*}}\left(\tilde{F}_{i}^{{\rm vir}}\right)}=\frac{\big(\tilde{i^{'}}\big)^*\tilde{f}_{*}\left[1\right]^{{\rm vir}}}{t\left(-t\right)^{m}}
\]
where $\tilde{f}:\overline{\mathcal{M}}_{g,1+m}\left(\mathbb{P}^{1},1\right)\rightarrow\overline{\mathcal{M}}_{g,1+m}\times\left(\mathbb{P}^{1}\right)^{m+1}$ and $\tilde{i^{'}}:\overline{\mathcal{M}}_{g,m+1}\times\left\{ 0\right\} \times\left\{ \infty\right\} ^{m}\hookrightarrow\overline{\mathcal{M}}_{g,1+m}\times\left(\mathbb{P}^{1}\right)^{m+1}$. To each fixed point of the action is associated a graph, the fixed locus $\tilde{F}_{i}$, for $0\leq i\leq g$, corresponds to the graph with two vertices connected by a degree $1$ edge, such that one vertex is of genus $i$ with the first marked point on it and associated to $0\in \mathbb{P}^{1}$, and the other vertex is of genus $g-i$ with the $m$ remaining marked points and associated to $\infty \in \mathbb{P}^{1}$. In our case we have $e_{\mathbb{C}^{*}}\left(\overline{\mathcal{M}}_{g,m+1}\times\left\{ 0\right\} \times\left\{ \infty\right\} ^{m}\right)=t\left(-t\right)^{m}$ giving the denominator on the RHS. Now, we extract the coefficient of $\frac{1}{t^{r+m+2}}$ in this formula, since $\tilde{i^{'}}\tilde{f}_{*}\left[1\right]^{{\rm vir}}$ is a polynomial in $t$, the RHS vanishes. The contribution of the Euler class $e_{\mathbb{C}^{*}}\big(\tilde{F}_{i}^{{\rm vir}}\big)$ does not depend on the number of marked points $1+m\geq 3$, so  it is the same than in the case $1+m=3$ treated in \cite{AS09}. Thus, the localization formula gives exactly formula~(\ref{eq: AS any m}), where the first term corresponds to the contribution of the fixed locus $\tilde{F}_{0}$, and the second term to the fixed loci $\tilde{F}_{g_1},$ for $0<g_1\leq g$. To get formula~(\ref{eq: AS m=2}), it suffices to substitute $m=2$, and then to push forward this equality to $\overline{\mathcal{M}}_{g,2}$ by the forgetful map forgetting the point $3$, using the string equation when necessary. 
\end{proof}

\begin{proposition}
For any $r\geq0$ we have:
\begin{align}
0=\  & \left(-1\right)^{r-1}\DR_{g}(1,-1)\psi_{1}^{r}|_{\oM_{g,2}}+\Big(\frac{\Lambda_g}{1-\psi_{2}}\Big|_{\oM_{g,2}}\Big)_{g+r}\label{eq: Arcara with DR} \\\notag
 & +\sum_{\substack{g_{1}+g_{2}=g\\
g_{1},g_{2}\geq1\\
a_{1}+a_{2}=g_{1}+r-1
}
}(-1)^{a_{2}-1}\Big(\frac{\Lambda_{g_1}}{1-\psi'}\Big|_{\oM_{g_{1},2}}\Big)_{a_{1}}\diamond (\psi'')^{a_{2}}\DR_{g_{2}}(1,-1)|_{\oM_{g_{2},2}}.
\end{align}
Here $\psi'$ and $\psi''$ again denote the $\psi$-classes at the two branches of the node formed by $\diamond$. 
\end{proposition}

\begin{proof}
The proof of this formula is obtained by localization on the moduli space of degree $1$ stable relative maps to $\left(\mathbb{P}^{1},\infty\right)$. The localization formula for moduli of relative stable maps was obtained by Graber and Vakil in \cite{graber2005relative}.

Recently, the localization on the moduli space of degree $1$ stable relative maps to $\left(\mathbb{P}^{1},\infty\right)$ was used to obtain a formula analogous to Eq.~(\ref{eq: Arcara with DR}), see \cite[Theorem 5.3]{BS22}. The proof of Eq.~(\ref{eq: Arcara with DR}) is almost identical to the proof of \cite[Theorem 5.3]{BS22}. In this context, we will only point out the necessary modifications in the latter proof to derive a valid proof for Eq.~(\ref{eq: Arcara with DR}), while referring to \emph{op.~cit.} for notation and details. It worth noting that  \cite[Appendix]{BS22} provides a detailed description of the virtual localization techniques applied to the moduli space of stable relative maps to $\left(\mathbb{P}^{1},\infty\right)$. 

The proof of \cite[Theorem 5.3]{BS22} goes as follows: first define a $\mathbb{C}^{*}$-action on the moduli space of degree $1$ stable relative maps to $\left(\mathbb{P}^{1},\infty\right)$ denoted by $\overline{\mathcal{M}}_{g,1}\left(\mathbb{P}^{1},1\right)$. Then, use the localization formula to express the virtual fundamental class of $\overline{\mathcal{M}}_{g,1}\left(\mathbb{P}^{1},1\right)$ as sum of contributions associated to each fixed locus of the $\mathbb{C}^{*}$-action in the equivariant cohomology of $\overline{\mathcal{M}}_{g,1}\left(\mathbb{P}^{1},1\right)$. Then, intersect this formula with a $\mathbb{C}^{*}$-equivariant class
\[
I_{g}={\rm ev}_{2}^{*}\left(\left[0\right]\right)e_{\mathbb{C}^{*}}\left(B\right),
\]
(here $B$ is an auxiliary rank $g$ equivariant vector bundle designed to kill the numerator of the localization formula; this idea goes back to~\cite{Liupan})
and push this equality to the moduli space of curves $\overline{\mathcal{M}}_{g,2}$. Finally, extract the coefficient of $u^{-1-r}$ for $r\geq0$, where $u$ is the generator of the equivariant cohomology of the point. This kills the contribution coming from virtual fundamental class of $\overline{\mathcal{M}}_{g,1}\left(\mathbb{P}^{1},1\right)$  since this class is equivariant, and this vanishing proves the formula in~\emph{op.~cit}. 

In order to prove our formula, it suffices to replace the class $I_{g}$ by
\[
\tilde{I}_{g}:={\rm ev}_{2}^{*}\left(\left[0\right]\right).
\]
The argument works in the same way: intersecting the localization formula with $\tilde{I}_{g}$ selects the same fixed loci of the action. However, the intersection of each fixed contribution of the localization formula with $\tilde{I}_{g}$ is of course different. We now explicitly give the contribution of each fixed locus selected by $\tilde{I}_{g}$. Using the notations of \cite[Theorem 5.3]{BS22} we have:
\begin{itemize}
	\item $\iota_{0}^{*}\left(\tilde{I}_{g}\right)=u$, then 
	\[
	\iota_{0}^{*}\left(\frac{\tilde{I}_{g}}{e_{\mathbb{C}^{*}}\left(N_{0}^{{\rm vir}}\right)}\right)=\frac{\sum_{i=0}^{g}\left(-1\right)^{i}\lambda_{i}u^{g-i}}{u-\psi_{1}}
	\]
	and thus 
	\[
	\epsilon_{*}\iota_{0*}\left(\iota_{0}^{*}\left(\frac{\tilde{I}_{g}}{e_{\mathbb{C}^{*}}\left(N_{0}^{{\rm vir}}\right)}\right)\cap\left[\overline{\mathcal{M}}_{g,2}\right]\right)=\sum_{a\geq0}u^{g-a-1}\left(\frac{\Lambda_g}{1-\psi_{1}}\right)_{a};
	\]
	\item $\iota_{\Gamma_{0,g}}^{*}\left(\tilde{I}_{g}\right)=u$, then
	\[
	\iota_{\Gamma_{0,g}}^{*}\left(\frac{\tilde{I}_{g}}{e_{\mathbb{C}^{*}}\left(N_{\Gamma_{0,g}}^{{\rm vir}}\right)}\right)=\frac{1}{-u-\tilde{\psi}_{0}}
	\]
	and thus 
	\[
	\epsilon_{*}\iota_{\Gamma_{0,g}*}\left(\iota_{\Gamma_{0,g}}^{*}\left(\frac{\tilde{I}_{g}}{e_{\mathbb{C}^{*}}\left(N_{\Gamma_{0,g}}^{{\rm vir}}\right)}\right)\cap\left[\overline{\mathcal{M}}_{\Gamma_{0,g}}\right]^{{\rm vir}}\right)=\sum_{i\geq0}u^{-i-1}\left(-1\right)^{i-1}\psi_{2}^{i}{\rm DR}_{g}\left(-1,1\right);
	\]
	\item 
	$\iota_{\Gamma_{g_{1},g_{2}}}^{*}\left(\tilde{I}_{g}\right)=u$, then
	\[
	\iota_{\Gamma_{g_{1},g_{2}}}^{*}\left(\frac{\tilde{I}_{g}}{e_{\mathbb{C}^{*}}\left(N_{\Gamma_{g_{1},g_{2}}}^{{\rm vir}}\right)}\right)=\frac{\sum_{i=1}^{g_{1}}\left(-1\right)^{i}\lambda_{i}u^{g_{1}-i}}{\left(u-\psi_{1}\right)}\times\frac{1}{\left(-u-\tilde{\psi}_{0}\right)},
	\]
	and thus 
	\begin{align*}
		& \epsilon_{*}\iota_{\Gamma_{g_{1},g_{2}}}\left(\iota_{\Gamma_{g_{1},g_{2}}}^{*}\left(\frac{\tilde{I}_{g}}{e_{\mathbb{C}^{*}}\left(N_{\Gamma_{g_{1},g_{2}}}^{{\rm vir}}\right)}\right)\cap\left[\overline{\mathcal{M}}_{\Gamma_{g_{1},g_{2}}}\right]^{{\rm vir}}\right)\\ 
		& =\sum_{a_{1},a_{2}\geq0}u^{g_{1}-a_{1}-a_{2}-2}\left(-1\right)^{a_{2}-1}\times\left(\frac{\Lambda_{g_1}}{1-\psi'}\right)_{a_{1}}\times(\psi'')^{a_{2}}{\rm DR}_{g_{2}}\left(-1,1\right),
	\end{align*}
\end{itemize} 
cf. Eqs~(5.18),~(5.19), and~(5.20) in~\cite{BS22}, respectively. We extract the coefficient of $u^{-1-r}$ in each of these three terms. Their sum vanishes as we outlined above, and this vanishing precisely proves Eq.~(\ref{eq: Arcara with DR}).
\end{proof}

\subsection{Proofs of Theorems~\ref{thm:anygn1n0-mgeq2} and~\ref{thm:anyg-n0n1-m1}}

We only have to treat the $n=1$ case of both theorems as the $n=0$ case is proved in Section~\ref{sec:n0-case}. To this end, multiply Equations~(\ref{eq: AS m=2}),~(\ref{eq: AS any m}), and~(\ref{eq: Arcara with DR}) by $\lambda_g$ and use Propositions~\ref{prop:reducltion-to-lambda} and~\ref{prop:reducltion-to-lambdan0} to replace in all three equations the classes of the type
\[
\frac{\lambda_{\tilde g}\Lambda_{\tilde g}}{1-\psi_1} \Big|_{\oM_{\tilde{g},1+\tilde m}} 
\]
in all instances where they occur by $\Omega_{\tilde g,1+\tilde m}(1,0;-1,0,\dots,0)$ (we mean that $(\tilde g,\tilde m)$ can be $(g,2)$, or $(g_1,2)$, or $(g_2,m)$, and sometimes we have to relabel the marked points to call the one with the $\psi$-class the first one). We have:
\begin{align}
	\label{eq:AS-Omega} 0 = \ & - \big(\lambda_g\Omega_{g,2}(1,0;-1,0)\big)_{2g+r}
	+ (-1)^{r} \big(\lambda_g\Omega_{g,2}(1,0;0,-1)\big)_{2g+r}
	\\ \notag
	& + \sum_{\substack{g_1+g_2=g\\ g_1,g_2\geq 1 \\ a_1+a_2 =2g-1+r}} (-1)^{g_1+a_1} 
	\big(\lambda_{g_1}\Omega_{g_1,2}(1,0;0,-1)\big)_{a_1}
	\diamond 
	\big(\lambda_{g_2}\Omega_{g_2,2}(1,0;-1,0)\big)_{a_2}. 
\\ \notag 
	 0 = \ & - \big(\lambda_g\Omega_{g,2}(1,0;-1,\underbrace{0,\dots,0}_m)\big)_{2g+m-1+r}
	\\ \notag
	& + \sum_{\substack{g_1+g_2=g\\ g_1\geq 1, g_2\geq 0\geq \\ a_1+a_2 =2g+m-2+r}} (-1)^{g_1+a_1} 
	\big(\lambda_{g_1}\Omega_{g_1,2}(1,0;0,-1)\big)_{a_1}
	\diamond 
	\big(\lambda_{g_2}\Omega_{g_2,2}(1,0;-1,\underbrace{0,\dots,0}_m)\big)_{a_2}. 
\\ \notag
	0=\  & \left(-1\right)^{r-1}\lambda_g \DR_{g}(1,-1)\psi_{1}^{r}|_{\oM_{g,2}}+\big(\lambda_g\Omega_{g,2}(1,0;0,-1)\big)_{2g+r} \\\notag
	& +\sum_{\substack{g_{1}+g_{2}=g\\
			g_{1},g_{2}\geq1\\
			a_{1}+a_{2}=2g_{1}+r-1
		}
	}(-1)^{a_{2}-1}\big(\lambda_{g_1}\Omega_{g_1,2}(1,0;0,-1)\big)_{a_1}\diamond (\psi'')^{a_{2}}\DR_{g_{2}}(1,-1)|_{\oM_{g_{2},2}}.
\end{align}
This way Equations~(\ref{eq: AS m=2}),~(\ref{eq: AS any m}), and~(\ref{eq: Arcara with DR}) multiplied by $\lambda_g$ are absolutely analogous to the Liu--Pandharipande equations in~\cite{Liupan} (see also~\cite[Eqs.~(5.1) and (5.2)]{BS22}) and~\cite[Theorem 5.3, Equation (5.16)]{BS22}, where, however, $\Psi_{\tilde g,1+\tilde m} (-1,0,\dots,0) = 1/(1-\psi_1)|_{\oM_{\tilde g, 1+\tilde m}}$ are used instead of the $\Omega$-classes multiplied by $\lambda_{\tilde g}$:
\begin{align}
	\label{eq:AS-Psi} 0 = \ & - \big(\Psi_{g,2}(-1,0)\big)_{2g+r}
	+ (-1)^{r} \big(\Psi_{g,2}(0,-1)\big)_{2g+r}
	\\ \notag
	& + \sum_{\substack{g_1+g_2=g\\ g_1,g_2\geq 1 \\ a_1+a_2 =2g-1+r}} (-1)^{g_1+a_1} 
	\big(\Psi_{g_1,2}(0,-1)\big)_{a_1}
	\diamond 
	\big(\Psi_{g_2,2}(-1,0)\big)_{a_2}. 
	\\ \notag 
	0 = \ & - \big(\Psi_{g,2}(-1,\underbrace{0,\dots,0}_m)\big)_{2g+m-1+r}
	\\ \notag
	& + \sum_{\substack{g_1+g_2=g\\ g_1\geq 1, g_2\geq 0\geq \\ a_1+a_2 =2g+m-2+r}} (-1)^{g_1+a_1} 
	\big(\Psi_{g_1,2}(0,-1)\big)_{a_1}
	\diamond 
	\big(\Psi_{g_2,2}(-1,\underbrace{0,\dots,0}_m)\big)_{a_2}. 
	\\ \notag
	0=\  & \left(-1\right)^{r-1}\lambda_g \DR_{g}(1,-1)\psi_{1}^{r}|_{\oM_{g,2}}+\big(\Psi_{g,2}(0,-1)\big)_{2g+r} \\\notag
	& +\sum_{\substack{g_{1}+g_{2}=g\\
			g_{1},g_{2}\geq1\\
			a_{1}+a_{2}=2g_{1}+r-1
		}
	}(-1)^{a_{2}-1}\big(\Psi_{g_1,2}(0,-1)\big)_{a_1}\diamond (\psi'')^{a_{2}}\DR_{g_{2}}(1,-1)|_{\oM_{g_{2},2}}.
\end{align}

 \noindent
Now we can use the same idea as we used in the proof of Theorem~\ref{thm:equivalence-conjectures-m-geq-2}:
starting from Eqs.~\eqref{eq:AS-Psi}, one obtains in a purely combinatorial way that uses only dimensions of the respective classes the ${}^{lvl}\Psi_{g,n}^m$-analogs of $n=1$ statements of Theorems~\ref{thm:anygn1n0-mgeq2} and~\ref{thm:anyg-n0n1-m1}, that is, \cite[Theorem 2.2]{BS22}, see Section 5 of \emph{op.~cit.} 
Hence, once we use exactly the same argument with $\Psi_{\tilde g,1+\tilde m}(-1,0,\dots,0)$ 
replaced by $\lambda_{\tilde g}\Omega_{\tilde g,1+\tilde m}(1,0;-1,0,\dots,0)$, we immediately obtain the proofs of the $n=1$ cases of Theorems~\ref{thm:anygn1n0-mgeq2} (through Theorem~\ref{thm:equivalence-conjectures-m-geq-2}) and~\ref{thm:anyg-n0n1-m1}.


\end{document}